\documentclass[sn-mathphys-num]{sn-jnl}

\usepackage{graphicx}%
\usepackage{multirow}%
\usepackage{amsmath,amssymb,amsfonts}%
\usepackage{amsthm}%
\usepackage{mathrsfs}%
\usepackage[title]{appendix}%
\usepackage{xcolor}%
\usepackage{textcomp}%
\usepackage{manyfoot}%
\usepackage{booktabs}%
\usepackage{algorithm}%
\usepackage{algorithmicx}%
\usepackage{algpseudocode}%
\usepackage{listings}
\usepackage{latexsym}
\usepackage[all, cmtip]{xy}

\theoremstyle{thmstyleone}%
\newtheorem{theorem}{Theorem}
\theoremstyle{thmstyletwo}%
\newtheorem{lemma}{Lemma}
\theoremstyle{thmstylethree}%
\newcommand{\gf}{ {{\mathbb F}} }

\begin{document}

\title[Permutation polynomials and their compositional inverses]{The compositional inverses of three classes of permutation polynomials over finite fields}


\author*[1]{\fnm{} \sur{Danyao Wu}}\email{wudanyao111@163.com}

\author[2]{\fnm{} \sur{Pingzhi Yuan}}\email{yuanpz@scnu.edu.cn}


\affil*[1]{\orgdiv{School of Computer Science and Technology}, \orgname{Dongguan University of Technology}, \orgaddress{
		\city{Dongguan}, \postcode{523808}, 
		\country{China}}}

\affil[2]{\orgdiv{School of Mathematics}, \orgname{South China Normal University}, \orgaddress{
		\city{Guangzhou}, \postcode{510631},
		\country{China}}}


\abstract{Recently, P. Yuan presented a local method to find permutation polynomials and their compositional inverses over finite fields. The work of P. Yuan inspires us to compute the compositional inverses of three classes of the permutation polynomials: (a) the permutation polynomials of the form $ax^q+bx+(x^q-x)^k$ over $\gf_{q^2},$ where $a+b \in \gf_q^*$ or $a^q=b;$ (b) the permutation polynomials of the forms $f(x)=-x+x^{(q^2+1)/2}+x^{(q^3+q)/2} $ and $f(x)+x$ over $\gf_{q^3};$ (c) the permutation polynomial of the form	$A^{m}(x)+L(x)$
	 over $\gf_{q^n},$
	where ${\rm Im}(A(x))$ is a vector space with dimension $1$ over $\gf_{q}$ and $L(x)$ is not a linearized permutation polynomial.}




\maketitle

\section{Introduction}

Let  $\gf_q$ be the finite field with $q$ elements and $\gf_{q}^*$ denote the
multiplicative group with the nonzero element in $\gf_q$, where $q$ is a prime power. Let $\gf_q[x]$
be the ring of polynomials in a single indeterminate $x$ over $\gf_q$. A polynomial
$f \in\gf_q[x]$ is called a {\em permutation polynomial} (PP) of $\gf_q$ if its
associated polynomial mapping $f: c\mapsto f(c)$ from $\gf_q$ to itself is  bijective. A polynomial $f(x)\in \gf_q[x]$ is called a complete permutation polynomial (CPP) if both $f(x)$ and $f(x)+x$ are permutations of $\gf_q.$  The unique polynomial denoted by $f^{-1}(x)$ over $\gf_q$
such that $f(f^{-1}(x))\equiv f^{-1}(f(x)) \equiv x \pmod{x^q-x}$ is called the compositional inverse of $f(x).$ Furthermore,  $f(x)$ is called  an involution when $f^{-1}(x)=f(x).$

The study of permutation polynomials and their compositional inverses over finite fields
in terms of their coefficients is a classical and difficult subject which
attracts people's interest partially due to their wide applications in coding theory
\cite{ding2013cyclic,ding2014binary,laigle2007permutation},
cryptography \cite{rivest1978method,schwenk1998public}, combinatorial design theory \cite{ding2006family}, and other areas of mathematics and engineering \cite{lidl1997finite,lidl1994introduction}. For instance, in block ciphers,
a permutation polynomial is usually used as an S-box to build the confusion layer and the compositional inverse of S-box comes into picture while decrypting the cipher. Both the permutation polynomial and its compositional inverse are implemented. Therefore, the explicit and efficient permutation polynomial and its compositional inverse are desired for designers. Indeed,
the better understanding of permutation polynomials and their compositional inverses in explicit format is not only meaningful, but also important for these applications.

In general, it is difficult to discover new classes of permutation polynomials. In 2011,  Akbrary, Ghioca and Wang \cite{akbary2011constructing} proposed a powerful method
called the AGW criterion for constructing permutation polynomials--that is, construct mappings $\bar{\lambda}(x)$, $\lambda(x),$ $h(x)$ that satisfy certain condition such that  $\overline{\lambda}(x)\circ f(x)=h(x)\circ \lambda(x),$ to discuss whether $f(x)$ is a permutation polynomial. Many classes of permutation polynomials have been constructed up  , see
\cite{akbary2011constructing,cepak2017permutations,li2017new,tu2018two,wu2022some,wu2022further,yuan2011permutation,yuan2015permutation,zheng2016large,zheng2019two}.
Computing the coefficients of
the compositional inverse of a permutation polynomial seems to be even more difficult, except for several classical classes such as
monomials, linearized polynomials, Dickson polynomials, which have nice structure. 
Wang \cite{wang2024surveyDCC}  surveyed on the  results and methods in the study of compositional
inverses of permutation polynomials over finite fields. More recent results are included in \cite{wu2024compositionalinversesclassespermutation,wu2024compositionalinversespermutationpolynomials,2wu2024compositionalinversespermutationpolynomials,wu2024permutation,wu2024permutationarXivlocal,yuan2022compositional}  for more details.

 The coauthor  gave  the dual diagram of the AGW criterion \cite{yuan2022compositional} and  a local method \cite{yuan2022local} to find permutation polynomials and their  compositional inverses, which the latter result  proposed the framework to obtain the permutation polynomials and their compositional inverses.  This inspires us to merge these two conclusions to study the compositional inverses of permutation polynomials, as established by the AGW criterion. In this paper, we will examine the compositional inverses of permutation polynomials of the form $ax^q+bx+(x^q-x)^k$ over $\gf_{q^2}$ as discussed in \cite{yuan2011permutation}.
 
Furthermore, the authors \cite{wu2024compositionalinversesclassespermutation} considered three classes of permutation trinomials  over $\gf_{q^3}$ using the local method. Using the techniques outlined in \cite{wu2024compositionalinversesclassespermutation}, we will investigate the compositional inverse of permutation polynomials of the form  $f(x)=-x+x^{(q^2+1)/2}+x^{(q^3+q)/2}$ over $\gf_{q^3}.$
 
  Based on the local method,  the authors \cite{wu2024permutation} examined a class of permutation polynomials of the form \begin{equation}\label{eqA(x)}
 	A_1^{m_1}(x)+\sum_{i=2}^nu_iA_i^{m_i}(x)
 	\end{equation}
 	 and their compositional inverses over $\gf_{q^n}$, where ${\rm Im}(A_i(x))$ is a vector space with dimension $1$ over $\gf_{q}$ and $u_i \in \gf_{q^n}$, $m_i$ are a positive integers for $i=1, 2, \cdots, n.$   	 
Hasan and Kaur \cite{hasan2024newclassespermutationpolynomials} 
considered six specific forms  given by  \eqref{eqA(x)} over $\gf_{q^3},$ setting $m_1=m_2=u_2=u_3=1$ and $A_3(x)=x+x^q+x^{q^2}.$ In this case,  we have $\sharp {\rm Im} (A_1(x)+A_2(x))\leq \sharp {\rm Im}(A_1(x))\cdot \sharp {\rm Im}(A_2(x))=q^2 $, indicating that   $A_1(x)+A_2(x)$ does not permute $\gf_{q^3}.$  In this paper, we refine 
the result of Hasan and Kaur \cite{hasan2024newclassespermutationpolynomials} by considering a class of permutation polynomials of the form 
\begin{equation*}\label{eqAL(x)}
	A^{m}(x)+L(x),
\end{equation*}
 and its compositional inverse over $\gf_{q^n},$
where ${\rm Im}(A(x))$ is a vector space with dimension $1$ over $\gf_{q}$ and $L(x)$ is not a linearized permutation polynomial.

The remainder of this paper is organized as follows. In Section 2, some basic concepts are introduced and  the local method is presented which is used in the sequels. In Section 3, we  the compositional inverses of permutation polynomials of the form $ax^q+bx+(x^q-x)^k$ over $\gf_{q^2}.$ In Section 4,  we investigate the compositional inverses of the permutation polynomials of the forms $f(x)=-x+x^{(q^2+1)/2}+x^{(q^3+q)/2} $ and $f(x)+x$ over $\gf_{q^3}.$ In Section 5, we study the permutation behavior of the form  	$A^{m}(x)+L(x)$
and its compositional inverse over $\gf_{q^n},$
where $Im(A(x))$ is a vector space with dimension $1$ over $\gf_{q}$ and $L(x)$ is not a linearized permutation polynomial.

\section{Preliminaries} 
%
%

The following lemma was developed by Akbary, Ghioca and Wang \cite{akbary2011constructing}. 
\begin{lemma}\label{leagw}\cite[AGW criterion]{akbary2011constructing}
	Let $A, S$ and $\overline{S}$ be finite sets
	with $\sharp S =\sharp \overline{S}$,
	and let $f(x) : A\longrightarrow A$, $h(x): S\longrightarrow \overline{S}$, $\lambda(x): A\longrightarrow S,$ and $\overline{\lambda}(x):A\longrightarrow \overline{S}$ be maps
	such that $\overline{\lambda}(x)\circ f(x)=h(x)\circ \lambda(x).$
	If both $\lambda(x)$ and $\overline{\lambda}(x)$ are surjective,
	then the following statements are equivalent:\\
	(i) $f(x)$ is bijective (a permutation of $A$); and\\
	(ii) $h(x)$ is bijective from $S$ to $\overline{S}$ and
	$f(x)$ is injective on $\lambda^{-1}(s)$ for each $s \in S.$
\end{lemma}
$$	\xymatrix{
	A  \ar[r]^{f(x)} \ar[d]_{\lambda(x)} & A  \ar[d]^{\bar{\lambda}(x)}\\
	S \ar[r]_{h(x)} &  \bar{S}
}
$$

We give the dual diagram of the AGW criterion.

\begin{lemma}\cite[Theorem 2.6]{yuan2022compositional}\label{leagwf-}
	Let the notations be defined as in Lemma \ref{leagw}. If $f(x) : A\longrightarrow A$ is a bijection, $f^{-1}(x)$ and $h^{-1}(x)$ are the compositional inverses of $f(x)$ and $h(x),$ respectively, then we have $$\lambda(x) \circ f^{-1}(x)=h^{-1}(x)\circ \bar{\lambda}(x),$$
	i.e., the following diagram commutes
	$$	\xymatrix{
		A  \ar[r]^{f^{-1}(x)} \ar[d]_{\bar{\lambda}(x)} & A  \ar[d]^{\lambda(x)}\\
		\bar{S} \ar[r]_{h^{-1}(x)} &  S
	}
	$$
	
\end{lemma}

Yuan \cite{yuan2022local} presented a local method to find the permutation polynomials and their  compositional inverses  over finite fields. We will use the local method frequently.
\begin{lemma}\label{leff-}\cite[Theorem 2.2]{yuan2022local}
	Let $q$ be a prime power and $f(x)$ be a polynomial over $\gf_q.$ Then
	$f(x)$ is a permutation polynomial over $\gf_q$ if and only if there exist nonempty finite
	subsets $S_i$, $i=1, 2, \cdots, t$ of $\gf_q$ and maps $\psi_i(x): \gf_q \rightarrow S_i$, $i=1, 2, \cdots, t$ such that $\psi_i(x)\circ f(x)=\varphi_i(x),$ $i=1, 2, \cdots, t$
	and $x=F(\varphi_1(x), \varphi_2(x), \cdots, \varphi_t(x)),$ where $F(x_1, x_2, \cdots, x_t)\in \gf_q[x_1, x_2, \cdots, x_t].$ Moreover, the compositional inverse of $f(x)$ is given by
	$$f^{-1}(x)=F(\psi_1(x), \psi_2(x), \cdots, \psi_t(x)).$$
\end{lemma}

Next, we will give a result about linearized permutation polynomial over $\gf_{q^n}.$
It is well-known that $L(x)=\sum_{i=0}^{n-1}a_ix^{q^i}\in \gf_{q^n}[x]$ is a permutation polynomial if and only if the associated Dickson matrix 

\begin{equation*}
	D=\left(
	\begin{array}{cccc}
		a_0&a_1&\cdots&a_{n-1}\\
		a_{n-1}^{q}&	a_{0}^{q}&\cdots&	a_{n-2}^{q}\\
		\vdots&\vdots&\vdots&\vdots\\
		a_{1}^{q^{n-1}}& 	a_{2}^{q^{n-1}}&\cdots& 	a_{0}^{q^{n-1}}\\
	\end{array}
	\right)
\end{equation*}
is non-singular.

 Wu and Liu  \cite{wu2013linearized} obtained the compositional inverse of a linearized permutation polynomial over $\gf_{q^n}.$ Recently, the authors  also obtained such result in \cite{wu2024permutationarXivlocal}  by the local method.
\begin{lemma}\cite{wu2013linearized, wu2024permutationarXivlocal}\label{linearinverse}
	Let $	L(x)=\sum_{i=0}^{n-1}a_ix^{q^i} \in \mathscr{D}_n(\gf_{q^n})$ be a linearized permutation polynomial and $D_L$ be its associated Dickson matrix. Then 
	$$L^{-1}(x)=\big(det(D_L)\big)^{-1}\sum_{i=0}^{n-1}\bar{a}_ix^{q^i},$$
	where $\bar{a}_i$ is the $(i ,0)-$th cofactor of $D_L$, and the determinant of  $D_L$ is $det(D_L)=a_0\bar{a}_0+\sum_{i=1}^{n-1}a_{n-i}^{q^i}\bar{a}_i.$ 
\end{lemma}

\section{The compositional inverses of the PPs of the form $ax^q+bx+(x^q-x)^k$ over $\gf_{q^2}$}

Yuan \cite{yuan2011permutation} studied the permutation behavior of the polynomials of the form $ax^q+bx+(x^q-x)^k$ over $\gf_{q^2}$  by applying the AGW criterion and gave the following result. 
\begin{theorem}\cite[Theorem 6.4]{yuan2011permutation}\label{1th2011yuan}
	Let $q$ be a prime power. \\
	(a) If $k\geqslant2$ is an even integer or $k$ is odd and $q$ is even, then $f_{a, b,k}(x)=ax^q+bx+(x^q-x)^k,$ $a, b \in \gf_{q^2} $ with $a+b \in \gf_q^*$ permutes $\gf_{q^2}$ if and only if $b\neq a^q.$\\
	(b) If $k$ and $q$ are odd positive integers, then $f_{a, k}(x)=ax^q+a^qx+(x^q-x)^k,$ $a \in \gf_{q^2}^* $ and $a+a^q \neq0$ permutes $\gf_{q^2}$ if and only if $\gcd(k, q-1)=1.$
\end{theorem}
 We will compute the compositional inverse of  $f_{a, b,k}(x)$  and $f_{a, k}(x)$ in Theorem \ref {1th2011yuan} by the dual diagram of the AGW criterion and the local method. We give a lemma at first. 
\begin{lemma}\label{lemma xk-}
	Let $q$ be an odd prime power and k be an odd positive integer. Let $S=\{\alpha^q-\alpha \mid \alpha \in \gf_{q^2}\}$ and $\bar{S}= \{-2s^k \mid s \in S\}$ be two subsets of $\gf_{q^2}.$ Assume that the polynomial $h(x)=-2x^k$ is a polynomial from $S$ to $\bar{S}.$ If there exist two integers $u, v $ such that $uk+v(q-1)=1,$ then the compositional inverse of $h(x)$ from $\bar{S}$ to $S$ is given by 
	$$h^{-1}(x)=(-1)^v(-2)^{-u}x^u.$$ 
\end{lemma}
\begin{proof}
Since $\alpha^q-\alpha=0$ if  $\alpha \in \gf_q$, and $\alpha^q-\alpha\neq0$ if  $\alpha \in  \gf_{q^2}\setminus \gf_q, $  we have  $S=\{\alpha^q-\alpha \mid \alpha \in \gf_{q^2}\}=\{\alpha^q-\alpha \mid \alpha \in \gf_{q^2}\setminus \gf_q \}\cup\{0\}.$   Moreover, 
	Since $q$ is odd, we have \begin{equation} \label{1eqlemmayuan2011}
		(x^q-x)^{q-1}=
		\begin{cases}
			0, &\,\text {if $x \in \gf_q$ },\\
			-1	, &\,\text {if $ x \in  \gf_{q^2}\setminus \gf_q$ },		
		\end{cases}
	\end{equation} and so for any  $x \in \gf_{q^2}\setminus \gf_q$ ,  we obtain 
\begin{align*}
	\left((-1)^v(-2)^{-u}x^u\right)\circ h(x^q-x)=&\,\left((-1)^v(-2)^{-u}x^u\right) \circ \left(-2(x^q-x)^k\right)\\
	=&\,(-1)^v(x^q-x)^{uk}\\
	=&\, (x^q-x)^{v(q-1)+uk}\\
	=&\, x^q-x	
\end{align*}
as $uk+v(q-1)=1.$ Moreover, The value of $(-1)^v(-2)^{-u}x^u$  at $x=0$ is also equal to zero.
Hence, the compositional inverse of $h(x)$ is 
	$$h^{-1}(x)=(-1)^v(-2)^{-u}x^u.$$
We are done. 
\end{proof}
\begin{theorem}(i) Using the notations as in Theorem   \ref{1th2011yuan} (a). If the polynomial  $f_{a, b,k}(x)=ax^q+bx+(x^q-x)^k$ permutes $\gf_{q^2}$, then the compositional inverse of $f_{a, b,k}(x)$ over $\gf_{q^2}$ is given by 
	$$f_{a, b,k}^{-1}(x)=(a+b)^{-1}\left((b-a^q)^{-k}(x^q-x)^k-a(b-a^q)^{-1}(x^q-x)-x\right).$$ 
(ii) Using the notations as in Theorem \ref{1th2011yuan} (b). If the polynomial $f_{a, k}(x)=ax^q+a^qx+(x^q-x)^k$ permutes $\gf_{q^2},$  then the compositional inverse of $f_{a, k}(x)$ over $\gf_{q^2}$ is given by 
	$$f_{a, k}^{-1}(x)=(a+a^q)^{-1}\left((-1)^{kv}(-2)^{-ku}(x^q-x)^{ku}-(-1)^v(-2)^{-u}a(x^q-x)^{u}-x\right),$$  where  $u, v $ are intergers with $uk+v(q-1)=1.$
\end{theorem}
\begin{proof}
	(i) As it  have been shown in \cite{yuan2011permutation} that $$(x^q-x)\circ f_{a, b,k}=(b-a^q)x \circ (x^q-x),$$  
	we have  $$x^q-x=(b-a^q)^{-1}x \circ (x^q-x)\circ f_{a, b,k}$$ by  Lemma \ref{leagwf-}.
	
	Let $\psi_1(x)=x^q-x=(b-a^q)^{-1}x \circ (x^q-x)\circ f_{a, b,k},$ $\psi_2(x)=f_{a, b,k}(x)$, $\varphi_1(x)=(b-a^q)^{-1}(x^q-x),$ and $\varphi_2(x)=x.$
	Then we have the system of the equations 
	\begin{equation*}
		\begin{cases}
			ax^q+bx&=\psi_1^k(x)-\psi_2(x),\\
			x^q-x&=\psi_1(x).
		\end{cases}
	\end{equation*}
	By eliminating the indeterminate $x^q$ using the above system, we obtain $$x=(a+b)^{-1}\left(\psi_1^k(x)-\psi_2(x)-a\psi_1(x)\right).$$
	It follows from Lemma \ref{leff-} that the compositional inverse of$f_{a, b,k}(x)$ over $\gf_{q^2}$ is given by 
	$$f_{a, b,k}^{-1}(x)=(a+b)^{-1}\left((b-a^q)^{-k}(x^q-x)^k-a(b-a^q)^{-1}(x^q-x)-x\right).$$ 
	
	(ii) Let $S=\{\alpha^q-\alpha \mid \alpha \in \gf_{q^2}\},$  $\bar{S}= \{-2s^k \mid s \in S\}$ be two subsets of $\gf_{q^2}$ and $h(x)=-2x^k$ be a polynomial from $S$ to  $\bar{S}.$  As it have been shown in \cite{yuan2011permutation} that $$(x^q-x)\circ f_{a, k}=h(x) \circ (x^q-x), $$ we have   $$x^q-x=h^{-1}(x) \circ (x^q-x)\circ f_{a, k}$$ according to Lemma \ref{leagwf-}.
	
	Let $\psi_1(x)=x^q-x=h^{-1}(x) \circ (x^q-x)\circ f_{a, k},$ $\psi_2(x)=f_{a, k}(x)$, $\varphi_1(x)=h^{-1}(x) \circ (x^q-x),$ and $\varphi_2(x)=x.$
	Then we have the system of the equations 
	\begin{equation*}
		\begin{cases}
			ax^q+a^qx&=\psi_1^k(x)-\psi_2(x),\\
			x^q-x&=\psi_1(x).
		\end{cases}
	\end{equation*}
	By eliminating the indeterminate $x^q$ using the above system, we obtain $$x=(a+a^q)^{-1}\left(\psi_1^k(x)-\psi_2(x)-a\psi_1(x)\right).$$
	It follows from Lemmas \ref{leff-} and \ref{lemma xk-} that the compositional inverse of$f_{a, k}(x)$ over $\gf_{q^2}$ is given by 
	$$f_{a, k}^{-1}(x)=(a+a^q)^{-1}\left((-1)^{kv}(-2)^{-ku}(x^q-x)^{ku}-(-1)^v(-2)^{-u}a(x^q-x)^{u}-x\right).$$ 
\end{proof}

\section{The compositional inverse of a class of CPP over $\gf_{q^3}$}

Ma al et. \cite{ma2017some} studied a class of complete permutation polynomial and gave the following result.

\begin{theorem} \cite[Theorem 4.1]{ma2017some}\label{thmma2017some4.1}
	Let $q$ be an odd prime power,  $f(x)=-x+x^{(q^2+1)/2}+x^{(q^3+q)/2}$ is a CPP over $\gf_{q^3}.$
\end{theorem}

To compute the composition inverse of $f(x)$ and $f(x)$ in Theorem \ref{thmma2017some4.1},  we give a lemma first. 

\begin{lemma}\label{lemmmama1}
	Let $q$ be an odd prime. For any $a \in \gf_{q^3}^*,$ the unique solution of  the equation $a^qx^q+a^{q^2}x-2=0$ over $\gf_{q^3}$ is $x=a^{-(q^2+q+1)}(a^{q+1}-a^{2q}+a^{q^2+q}).$
\end{lemma}
\begin{proof}
	Let $A(x)=a^qx^q+a^{q^2}x-2=(x-2)\circ L(x)$, where $L(x)=a^qx^q+a^{q^2}x$ is a linearized polynomial over $\gf_{q^3}.$ Since the determinant of Dickson matrix of $L(x)$ is $2a^{q^2+q+1} \neq 0,$ we have that $L(x)$ permutes $\gf_{q^3},$ and so $A(x)$ permutes $\gf_{q^3},$ that is to say,  that the equation $a^qx^q+a^{q^2}x-2=0$ has a unique solution over $\gf_{q^3}.$ 
	Moreover, according to Lemma \ref{linearinverse}, we have that the compositional inverse of $L(x)$ over $\gf_{q^3}$ is  $$L^{-1}(x)=(2a^{q^2+q+1})^{-1}(a^{q+1}x-a^{2q}x^q+a^{q^2+q}x^{q^2}).$$
	This yields that 
	the compositional inverse of $A(x)$ over $\gf_{q^3}$ is 
	$$A^{-1}(x)=L^{-1}(x)\circ (x+2)=(2a^{q^2+q+1})^{-1}(a^{q+1}(x+2)-a^{2q}(x+2)^q+a^{q^2+q}(x+2)^{q^2}).$$
	Therefore, $A^{-1}(0)=a^{-(q^2+q+1)}(a^{q+1}-a^{2q}+a^{q^2+q}).$
	We complete the proof. 
\end{proof}

\begin{theorem}
	Using the notations as in Theorem \ref{thmma2017some4.1}, if 
	 $f(x)=-x+x^{(q^2+1)/2}+x^{(q^3+q)/2}$ is a CPP  over $\gf_{q^3}.$ 
	Then the compositional inverse of $f(x)+x$ over $\gf_{q^3}$ is $$\left(f(x)+x\right)^{-1}= \left((x-x^q+x^{q^2})/2\right)^{q^3-q^2+q}.$$
	and the compositional inverse of $f(x)$ over $\gf_{q^3}$ is 
	\begin{align*}
		f^{-1}(x)
		=&\, \left(x^{q^2+q+1}(x^{q+1}-x^{2q}+x^{q^2+q})^{q^3-2}\right)^{q^3-q^2+q}.
	\end{align*}
\end{theorem}
\begin{proof}
	Note that 
	\begin{equation}\label{1eqma1}
		f(x)+x=x^{(q^2+1)/2}+x^{(q^3+q)/2}=g(x^{(q^2+1)/2}),
	\end{equation}
	where are $g(x)=x+x^q$.
	
	Since $g(x)$ is a linearized polynomial over $\gf_{q^3}$ and the  determinant of Dickson matrix of $g(x)$ is $2$, it implies by Lemma \ref{linearinverse} that the compositional inverse of $g(x)$ over $\gf_{q^3}$ is 
	\begin{equation}\label{2eqma1}
		g^{-1}(x)=(x-x^q+x^{q^2})/2.
	\end{equation}
	Moreover, since 
	$(q^3-q^2+q)(q^2+1)/2-\left((q^2-q+2)/2\right)(q^3-1)=1$, it follows from Eqs. \eqref{1eqma1} and \eqref{2eqma1} that 
	\begin{align*}
		\left(f(x)+x\right)^{-1}=&\,\left(x^{(q^2+1)/2}\right)^{-1} \circ g^{-1}(x)\\
		=&\, \left((x-x^q+x^{q^2})/2\right)^{q^3-q^2+q}.
	\end{align*}
	
	Now we will study the compositional inverse of $f(x)$ over $\gf_{q^3}.$ Put $h(x)=x+x^q-x^{1+q-q^2}.$ Then we have 
	\begin{equation}\label{1eqma1f(x)}
		f(x)=h(x) \circ x^{(q^2+1)/2}
	\end{equation}
	Therefore, it is crucial to calculate the compositional inverse of $h(x)$ over $\gf_{q^3}.$ 
	
	Let $\psi_1(x)=x, \psi_2(x)=x^q, \psi_3(x)=x^{q^3}, $ $ \varphi_1(x)=\psi_1(x)\circ h(x)=h(x), \varphi_2(x)=\psi_2(x)\circ h(x)=h^q(x)$, and $  \varphi_3(x)=\psi_3(x)\circ h(x)=h^{q^2}(x).$ For simplicity, put $\varphi_1(x)=a, \varphi_2(x)=b, \varphi_3(x)=c.$
	Then  $c=b^q=a^{q^2}.$ As it has been shown in \cite{ma2017some}  that $f(x)$ has a unique  root $0$ in $\gf_{q^3},$  we  assume that $x\neq 0$  so that  $abc\neq0.$ By substituting $y=x^q,$ $z=y^q, $ we obtain the system of equations 
	\begin{equation*}
		\begin{cases}
			x+y+\frac{xy}{z}=a,\\
			y+z+\frac{yz}{x}=b,\\
			z+x+\frac{xz}{y}=c,\\
		\end{cases}
	\end{equation*}
	which can be rewritten as 
	\begin{equation}\label{1eqsma1}
		\begin{cases}
			xz+yz-xy=&az,\\
			xy+xz-yz=&bx,\\
			yz+xy-xz=&cy.\\	
		\end{cases}
	\end{equation}
	By adding the last two equations of Eq. \eqref{1eqsma1}, 
	we have $$2xy=bx+cy,$$ or, equivalently, $$2x^q=a^q+a^{q^2}x^{q-1}$$ because $c=b^q=a^{q^2}$, $y=x^{q}$ and $x\neq 0.$ Setting $t=1/x,$ we obtain $$a^qt^q+a^{q^2}t-2=0$$ by the above equation. According to Lemma 
	\ref{lemmmama1}, we have $t=a^{-(q^2+q+1)}(a^{q+1}-a^{2q}+a^{q^2+q}),$ and so $x=a^{q^2+q+1}(a^{q+1}-a^{2q}+a^{q^2+q})^{q^3-2}.$
	Hence, by Lemma \ref{leff-}, the compositional inverse of $f(x)$ over $\gf_{q^3}$ is 
	\begin{align*}
		f^{-1}(x)=&\,\left(x^{(q^2+1)/2}\right)^{-1} \circ h^{-1}(x)\\
		=&\, \left(x^{q^2+q+1}(x^{q+1}-x^{2q}+x^{q^2+q})^{q^3-2}\right)^{q^3-q^2+q}.
	\end{align*}
	We are done.
\end{proof}

\section{The permutation polynomial of the form $A(x)^m +L(x)$ and its inverse over $\gf_{q^n}$}

In this section, we investigate the permutation behavior of the form the form $A(x)^m +L(x)$ and its inverse over $\gf_{q^n}.$

We give a  property of linearized polynomial over $\gf_{q^n}$ at first.

\begin{lemma} \label{DL}\cite[Proposition 4.4] {wu2013linearized}
	For any linearize polynomial $L(x)$ over $\gf_{q^n}$, then $rank(L)=rank(D)$, where rank(L) is the rank of the linear transformation induced by $L(x),$ and $D$ is the Dickson matrix associated to $L(x).$ 
\end{lemma}
We investigate a property of  a linearized polynomial $L(x)$ over $\gf_{q^n}$ with $rank(L)=n-1$  in the following result. 
\begin{lemma}\label{linearly}
	Let	$L(x)=\sum_{i=0}^{n-1}a_ix^{q^i}\in \gf_{q^n}[x]$  and $D$ be the associated Dickson matrix. If $\text{rank}(D)=n-1,$ then any $n-1$ row vectors of the matrix $D$ are linearly independent. 
\end{lemma}
\begin{proof}
	Suppose, on the contrary,  that there exist $n-1$ row vectors being linear dependent. In fact, we can assume that the first n-1 rows vectors of matrix D are linearly dependent, and other cases can be similarly proved. Suppose that the  vectors $(a_0, a_1, \cdots, a_{n-1}),$  $(a_{n-1}^{q}, a_{0}^{q},\cdots, a_{n-2}^{q})$ ,$\cdots$,     $(a_{2}^{q^{n-2}}, a_{3}^{q^{n-2}},\cdots, a_{1}^{q^{n-2}})$ are linear dependent, and so we get a relation  
	\begin{equation}\label{a12}
		k_1\left({\begin{array}{c}a_0\\ a_1\\ \vdots\\ a_{n-1}\end{array}}\right)^T+k_2\left({\begin{array}{c}a_{n-1}^{q}\\ a_{0}^{q}\\ \vdots\\ a_{n-2}^{q}\end{array}}\right)^T+\cdots+k_n\left({\begin{array}{c}a_{2}^{q^{n-2}}\\ a_{3}^{q^{n-2}}\\ \vdots\\ a_{1}^{q^{n-2}}\end{array}}\right)^T=\left({\begin{array}{c}0\\ 0\\ \vdots\\ 0\end{array}}\right)^T
	\end{equation}
	with $k_i \in \gf_{q^n}$ not all being $0.$ 
	Raising Eq.(\ref{a12}) to the power of $q,$ and then rearranging the row vectors components suitably,  
	we obtain  
	\begin{equation}\label{a13}
		k_1^q\left({\begin{array}{c}
				a_{n-1}^{q}\\ a_{0}^{q}\\ \vdots \\ a_{n-2}^{q}
		\end{array}}\right)^T+k_2^q\left({\begin{array}{c}
				a_{n-2}^{q^2}\\ a_{n-1}^{q^2}\\ \cdots\\a_{n-3}^{q^3}
		\end{array}}\right)^T
		+\cdots+k_n^q\left({\begin{array}{c}
				a_{1}^{q^{n-1}}\\ a_{2}^{q^{n-1}}\\ \vdots\\ a_{0}^{q^{n-1}}
		\end{array}}\right)^T=\left({\begin{array}{c}0\\ 0\\ \vdots\\ 0\end{array}}\right)^T,
	\end{equation}
	which means  except for the first row vector of matrix $D$,  the last n-1 rows vectors of matrix D are linearly dependent. 
	Similarly,  raising Eq.(\ref{a12}) to the   $q^2$-th power and then properly permuting the positions of the row vectors components, we obtain 
	\begin{equation*}
		k_1^{q^2}\left({\begin{array}{c}
				a_{n-2}^{q^2}\\ a_{n-1}^{q^2}\\ \cdots\\a_{n-3}^{q^3}
		\end{array}}\right)^T
		+\cdots+k_{n-1}^{q^2}\left({\begin{array}{c}
				a_{1}^{q^{n-1}}\\ a_{2}^{q^{n-1}}\\ \vdots\\ a_{0}^{q^{n-1}}
		\end{array}}\right)^T+k_{n}^{q^2}\left({\begin{array}{c}
				a_0\\ a_1\\ \vdots\\ a_{n-1}
		\end{array}}\right)^T=\left({\begin{array}{c}0\\ 0\\ \vdots\\ 0\end{array}}\right)^T.
	\end{equation*}
	That is,  except for the vector in the second row , the remaining $ n-1$ row vectors of matrix $D$ are linearly dependent. 
	We keep repeating the steps by  raising Eq.(\ref{a13}) to the power $q^3, \cdots, q^{n-1}$ and then properly permuting  the positions of row vectors components. Then we have that any $n-1$ row vectors of the matrix $D$ are linearly dependent, which is a contraction to $\text{rank}(D)=n-1.$
	We are done. 		
\end{proof}

For  be a prime power $q$ and a positive integer $n>1$, let 
$b \in \gf_{q^n}$ with $b^{q^{n-1}+q^{n-2}+...+q^{2}+q+1}=1$ and the polynomial
$$	A(x)=x^{q^{n-1}}+bx^{q^{n-2}}+b^{1+q^{n-1}}x^{q^{n-3}}+...+
b^{1+\sum_{i=1}^{j}q^{n-i}}x^{q^{n-j-2}}+...+b^{1+\sum_{i=1}^{n-2}q^{n-i}}x.$$
We demonstrate  that ${\rm Im}(A(x))$ is a  vector space with dimension $1$.  Furthermore, we present  a necessary and sufficient condition for  ${\rm Im}(A(x))^m$ to  also be a  vector space with dimension $1,$ as shown in the following lemma.
\begin{lemma}\label{a} \cite[Lemma 2]{wu2024permutation}
	Let the notation $A(x)$ be defined as above, then
	
	(i) $A(x)^q=b^qA(x);$
	
	(ii) ${\rm Im}(A(x)) =\{c g^{qt} \mid c \in \gf_q\}$  is a one-dimensional vector space over  $\gf_{q}$, where $g\in \gf_{q^n}$ is a primitive element of $\gf_{q^n}$ and $t$ is a positive integer with $b=g^{(q-1)t}.$
	
	(iii)  For a positive integer $m,$ $ {\rm Im}(A^m(x))$ is a one-dimensional vector space over  $\gf_{q}$ if and only if
	$\gcd(m, q-1)=1.$
\end{lemma}
%
%
%
%
\begin{theorem}\label{th1}
	For  a prime power  $q$ and   positive integers $n>1$ and $m$, let $b \in \gf_{q^n}$ with $b^{q^{n-1}+q^{n-2}+...+q^{2}+q+1}=1$ and the polynomial
	$$	A(x)=x^{q^{n-1}}+bx^{q^{n-2}}+b^{1+q^{n-1}}x^{q^{n-3}}+...+
	b^{1+\sum_{i=1}^{j}q^{n-i}}x^{q^{n-j-2}}+...+b^{1+\sum_{i=1}^{n-2}q^{n-i}}x.$$	
	Assume that  $L(x)=\sum_{i=0}^{n-1}a_ix^{q^i}$ is a linearized polynomial over $\gf_{q^n}$. If $L(x)$ is not a permutation polynomial over  $\gf_{q^n}$, then the polynomial 
	$$f(x)=A^m(x)+L(x)$$
	is a permutation polynomial over $\gf_{q^n}$ if and only if  $\gcd(m, q-1)=1$, $ \text{rank}(D)=n-1,$ $\alpha_1+\alpha_2b^{qm}+\alpha_3b^{q^2m+qm}+\cdots+\alpha_nb^{m\sum_{i=1}^{n-1}q^i}\neq0$ and  the determinant of the matrix

	\begin{equation*}
		B=\left(
		\begin{array}{cccccc}
			b^{1+\sum_{i=1}^{n-2}q^{n-i}}&	a_0&a_{n-1}^q&a_{n-2}^{q^2}&\cdots&a_{2}^{q^{n-2}}\\
			b^{1+\sum_{i=1}^{n-3}q^{n-i}}&	a_1&	a_{0}^{q}&a_{n-1}^{q^2}&\cdots&	a_{3}^{q^{n-2}}\\
			\vdots&\vdots&\vdots&\vdots&\vdots&\vdots\\	
			b&	a_{n-2}& 	a_{n-3}^{q}&	a_{n-4}^{q^{2}}&\cdots& 	a_{0}^{q^{n-2}}\\
			1&	a_{n-1}& 	a_{n-2}^{q}&	a_{n-3}^{q^{2}}&\cdots& 	a_{1}^{q^{n-2}}\\
		\end{array}
		\right)
	\end{equation*} 
	is non-zero, where the matrix $D$ is the associate Dickson matrix of $L(x), $ and $(\alpha_1, \alpha_2, \cdots, \alpha_n)$ is a non-zero solution of the system of linear equations $D^Tx=O.$
\end{theorem}

\begin{proof}
	We first prove the necessity. Put  $ \text{rank}(D)=r.$ Since $L(x)$ is not a permutation polynomial over $\gf_{q^n}$, we have $ \text{rank}(D)\leq n-1 $ by Lemma \ref{DL}.
	
	If $\gcd(m, q-1)\neq1$ or $ \text{rank}(D)=r< n-1, $ then it follows from   lemma \ref{a} that 
	\begin{equation}
		\sharp {\rm Im}(f(x))\leqq 	\sharp {\rm Im}(A^m(x))\cdot q^r\leqq q \cdot q^r< q^n,
	\end{equation}
	and so $f(x)$ is not a permutation polynomial over $\gf_{q^n}.$
	
	Now we assume that $ \text{rank}(D)= n-1, $ and $\gcd(m, q-1)=1.$ 
	Suppose that the determinant of $B$ is zero.  it follows that n 
	column vectors of $B$ are linearly dependent. Additionally, since  the rank of $D$ is $n-1,$ by Lemma \ref{linearly} we have that the first $n-1$ row vectors of $D$ are linear independent. 
	Furthermore, since the transpose  of   the last $n-1$ column vectors  of the matrix $B$ form the first $n-1$ row vectors of the matrix $D$, these last $n-1$ column vectors  of the matrix $B$ are linear independent.
	Given that the determinant of $B$ is zero, it follows that the first column vector of $B$ can be expressed as a linear combination of the last $n-1$ column vectors  of the matrix $B.$  Therefore,   there exist scalars $ k_2, \cdots, k_n \in \gf_{q^n}$ not all zero, such that 
	\begin{equation*}\label{ab}
		\left({\begin{array}{c}
				b^{1+\sum_{i=1}^{n-2}q^{n-i}}\\ b^{1+\sum_{i=1}^{n-3}q^{n-i}}\\
				\vdots \\
				1 \end{array}}\right)=
		k_2\left({\begin{array}{c}
				a_0\\ a_1\\ \vdots\\ a_{n-1} \end{array}}\right)	 			
		+\cdots +k_n\left({\begin{array}{c}a_{2}^{q^{n-2}}\\	a_{3}^{q^{n-2}}\\ \vdots\\ a_{1}^{q^{n-2}}\end{array}}\right).
	\end{equation*}
	This implies 
	\begin{align}\label{ab}
		\nonumber
		&\,	\left({\begin{array}{c}
				b^{1+\sum_{i=1}^{n-2}q^{n-i}}\\ b^{1+\sum_{i=1}^{n-3}q^{n-i}}\\
				\vdots \\
				1 \end{array}}\right)^T\left({\begin{array}{c}
				x\\x^q\\
				\vdots \\
				x^{q^{n-1}} \end{array}}\right)\\\nonumber
		= &\,  \left(	k_2\left({\begin{array}{c}
				a_0\\ a_1\\ \vdots\\ a_{n-1} \end{array}}\right)^T	 			
		+\cdots +k_n\left({\begin{array}{c}a_{2}^{q^{n-2}}\\	a_{3}^{q^{n-2}}\\ \vdots\\ a_{1}^{q^{n-2}}\end{array}}\right)^T\right)\left({\begin{array}{c}
				x\\x^q\\
				\vdots \\
				x^{q^{n-1}} \end{array}}\right)\\\nonumber
		=&\, k_2	\left({\begin{array}{c}
				a_0\\ a_1\\ \vdots\\ a_{n-1} \end{array}}\right)^T\left({\begin{array}{c}
				x\\x^q\\
				\vdots \\
				x^{q^{n-1}} \end{array}}\right)	 			
		+\cdots +k_n	\left({\begin{array}{c}a_{2}^{q^{n-2}}\\	a_{3}^{q^{n-2}}\\ \vdots\\ a_{1}^{q^{n-2}}\end{array}}\right)^T\left({\begin{array}{c}
				x\\x^q\\
				\vdots \\
				x^{q^{n-1}} \end{array}}\right) .\\
	\end{align}
	Note that 
	\begin{equation}\label{eqx1}
		\left({\begin{array}{cccc}
				b^{1+\sum_{i=1}^{n-2}q^{n-i}},& b^{1+\sum_{i=1}^{n-3}q^{n-i}},&
				\cdots, &
				1 \end{array}}\right)\left({\begin{array}{cccc}
				x, &x^q, &\cdots,&x^{q^{n-1}} \end{array}}\right)^T=A(x),
	\end{equation}
	and 
	\begin{equation}\label{lx}
		\left({\begin{array}{c}
				L(x)\\	L(x)^q\\
				\vdots \\
				L(x)^{q^{n-1}} \end{array}}\right) =\left(
		\begin{array}{cccc}
			a_0&a_1&\cdots&a_{n-1}\\
			a_{n-1}^{q}&	a_{0}^{q}&\cdots&	a_{n-2}^{q}\\
			\vdots&\vdots&\vdots&\vdots\\
			a_{1}^{q^{n-1}}& 	a_{2}^{q^{n-1}}&\cdots& 	a_{0}^{q^{n-1}}\\
		\end{array}
		\right) \left({\begin{array}{c}
				x\\x^q\\
				\vdots \\
				x^{q^{n-1}} \end{array}}\right)=D\left({\begin{array}{c}
				x\\x^q\\
				\vdots \\
				x^{q^{n-1}} \end{array}}\right),
	\end{equation}	
	or we can rewrite as 
	
	\begin{equation}\label{eqx2}
		\left({\begin{array}{cccccc}
				a_{n-j}^{q^j},& 
				\cdots,  &a_{n-1}^{q^j}, &a_{0}^{q^j}, &\cdots, &a_{n-j-1}^{q^j}
		\end{array}}\right)^T \left({\begin{array}{cccc}
				x,&x^q,&\cdots&x^{q^{n-1}} \end{array}}\right)=L(x)^{q^{j}},
	\end{equation}
	for $j=1, 2, \cdots, n. $
	
	Substituting  \eqref{eqx1} and \eqref{eqx2} into   \eqref{ab} yields 
	$$A(x)=k_2L(x)+k_3L^q(x)+k_4L^{q^2}(x)+\cdots+k_nL^{q^{n-2}}(x),$$
	Consequently, we can conclude that 
	\begin{equation}\label{f1}
		f(x)=A^m(x)+L(x)=\left((k_2x+k_3x^q+k_4x^{q^2}+\cdots+k_nx^{q^{n-2}})^m+x\right)\circ L(x).
	\end{equation}
	Since $L(x)$ is not a permutation polynomial over $\gf_{q^n},$  it follows from  (\ref{f1}) that $f(x)$ is not a permutation polynomial.  
	
	Now we assume that the martrix of $B$ is non-singular.  
	Since the rank of $D$ is $n-1,$ the system of homogeneous linear equations 
	\begin{equation}\label{dx}
		D^Tx=O,
	\end{equation}
	has a nontrivial solutions. Let $(\alpha_1, \alpha_2, \cdots, \alpha_n)^T$ be a non-zero solution of (\ref{dx}). Then we have 
	\begin{equation}\label{eqd}
		(\alpha_1, \alpha_2, \cdots, \alpha_n)D=(0, 0, \cdots, 0).
	\end{equation}
	Let $	\psi_1(x)=\alpha_1x+ \alpha_2x^q+\cdots+\alpha_nx^{q^{n-1}},$  $\psi_2(x)=x,$ $\varphi_1(x)=A(x)$ and $\varphi_2(x)=f(x).$ 
	It follows from  \eqref{lx} and \eqref{eqd} that  
	\begin{align}\label{amx}
		\nonumber	&\, \psi_1(x)\circ f(x)\\ \nonumber
		=&\, (\alpha_1, \alpha_2, \cdots, \alpha_n)\left({\begin{array}{c}
				x\\x^q\\
				\vdots \\
				x^{q^{n-1}} \end{array}}\right)\circ A^m(x)+(\alpha_1, \alpha_2, \cdots, \alpha_n)\left({\begin{array}{c}
				x\\x^q\\
				\vdots \\
				x^{q^{n-1}} \end{array}}\right)\circ L(x)\\ \nonumber 
		=&\, (\alpha_1, \alpha_2, \cdots, \alpha_n)\left({\begin{array}{c}
				A^m(x)\\A^{mq}(x)\\
				\vdots \\
				A^{mq^{n-1}}(x)\end{array}}\right)+ (\alpha_1, \alpha_2, \cdots, \alpha_n)\left({\begin{array}{c}
				L(x)\\L(x)^q\\
				\vdots \\
				L(x)^{q^{n-1}} \end{array}}\right)\\ \nonumber
		=&\,A^m(x)(\alpha_1, \alpha_2, \cdots, \alpha_n)\left({\begin{array}{c}
				1\\b^{qm}\\
				\vdots \\
				b^{m\sum_{i=1}^{n-1}q^i}\end{array}}\right)+(\alpha_1, \alpha_2, \cdots, \alpha_n)D\left({\begin{array}{c}
				x\\x^q\\
				\vdots \\
				x^{q^{n-1}} \end{array}}\right)\\ 
		=&\,(\alpha_1+\alpha_2b^{qm}+\alpha_3b^{q^2m+qm}+\cdots+\alpha_nb^{m\sum_{i=1}^{n-1}q^i})A^m(x).\\ \nonumber
	\end{align}
	If $\alpha_1+\alpha_2b^{qm}+\alpha_3b^{q^2m+qm}+\cdots+\alpha_nb^{m\sum_{i=1}^{n-1}q^i}=0,$ then $f(x)$ is not a permutation polynomial over $\gf_{q^n}.$
	
	Suppose that $\alpha_1+\alpha_2b^{qm}+\alpha_3b^{q^2m+qm}+\cdots+\alpha_nb^{m\sum_{i=1}^{n-1}q^i}\neq0.$  For simplicity, we put $s=\alpha_1+\alpha_2b^{qm}+\alpha_3b^{q^2m+qm}+\cdots+\alpha_nb^{m\sum_{i=1}^{n-1}q^i}.$
	Since  $\gcd(m, q-1)=1$, there there exist integers $u$ and $v$ such 
	that $mu\equiv1+v(q-1) \pmod{q^n-1},$ and so by (\ref{amx}), that we get 
	\begin{equation*}
		b^{-qv}\left(s^{-1}(\psi_1(x)\circ f(x))\right)^u=A^{mu}(x)=A^{1+v(q-1)}(x)=A(x)=\varphi_1(x).
	\end{equation*}
	Thus, we obtain 
	\begin{equation}\label{LLL}
		\varphi_2(x)-\varphi_1^m(x)=L(x).
	\end{equation}
	On the other hand, since the determinant of $B$ is not zero, the system of non-homogeneous linear equations 
	\begin{equation} \label{eqbax}
		Bx=(1, 0, \cdots , 0)^T, 
	\end{equation}
	has a unique solution. Let $(\beta_1, \beta_2, \cdots, \beta_n)^T$ be the unique solution of the system of non-homogeneous linear equations (\ref{eqbax}). Then we have 
	\begin{equation}\label{eqbax1}
		(\beta_1, \beta_2, \cdots, \beta_n)B^T=(1, 0, \cdots, 0).
	\end{equation}
	Put $D=(\eta_1, \eta_2, \cdots, \eta_n)$ and $\beta^T=(\beta_2, \beta_3, \cdots, \beta_n, 0)$, where $\eta_i$ is the $i$-th column vector of the matrix $D$,  
	It follows from Eq. \eqref{eqbax1} that 
	\begin{align*}
		\nonumber&\,	\beta_1(b^{1+\sum_{i=1}^{n-2}q^{n-i}}, b^{1+\sum_{i=1}^{n-3}q^{n-i}}, \cdots, 1)+(\beta_2, \beta_3, \cdots, \beta_n, 0)D\\\nonumber
		=&\, \beta_1(b^{1+\sum_{i=1}^{n-2}q^{n-i}}, b^{1+\sum_{i=1}^{n-3}q^{n-i}}, \cdots, 1)+\beta^T(\eta_1, \eta_2, \cdots, \eta_n)\\\nonumber
		=&\, (\beta_1b^{1+\sum_{i=1}^{n-2}q^{n-i}}+\beta^T\eta_1, \beta_1b^{1+\sum_{i=1}^{n-3}q^{n-i}}+\beta^T\eta_2, \cdots, \beta_1+\beta^T\eta_n,)\\
		=&\,(\beta_1, \beta_2, \cdots, \beta_n)B^T\\
		=&\, (1, 0, \cdots, 0),
	\end{align*}
	and so by \eqref{lx}, that 
	\begin{align*}
		&\beta_1\varphi_1(x)+\left((	\beta_2x+	\beta_3x^q+\cdots+	\beta_nx^{q^{n-2}})\circ \left(\varphi_2(x)-\varphi_1^m(x)\right )\right)\\
		=	&\,\beta_1\varphi_1(x)+\left((	\beta_2x+	\beta_3x^q+\cdots+	\beta_nx^{q^{n-2}})\circ L(x)\right)\\
		=&\, \beta_1(b^{1+\sum_{i=1}^{n-2}q^{n-i}}, b^{1+\sum_{i=1}^{n-3}q^{n-i}}, \cdots, 1)\left({\begin{array}{c}
				x\\x^q\\
				\vdots \\
				x^{q^{n-1}} \end{array}}\right)+(\beta_2, \beta_3, \cdots, \beta_n, 0)\left({\begin{array}{c}
				L(x)\\L(x)^q\\
				\vdots \\
				L(x)^{q^{n-1}} \end{array}}\right)\\	
		=&\, \beta_1(b^{1+\sum_{i=1}^{n-2}q^{n-i}}, b^{1+\sum_{i=1}^{n-3}q^{n-i}}, \cdots, 1)\left({\begin{array}{c}
				x\\x^q\\
				\vdots \\
				x^{q^{n-1}} \end{array}}\right)+(\beta_2, \beta_3, \cdots, \beta_n, 0)D\left({\begin{array}{c}
				x\\x^q\\
				\vdots \\
				x^{q^{n-1}} \end{array}}\right)\\
		=&\,x.\\
	\end{align*}
	It follows from Lemma (\ref{leff-}) that the polynomial $f(x)$ permutes $\gf_{q^n}$ and the compositional inverse of $f(x)$ over $\gf_{q^n}$ is 
	\begin{multline*}
		f^{-1}(x)= \beta_1b^{-qv}s^{-u}\left(\alpha_1x+ \alpha_2x^q+\cdots+\alpha_nx^{q^{n-1}}\right)^u+\\
		(	\beta_2x+	\beta_3x^q+\cdots+	\beta_nx^{q^{n-2}})\circ \left(x-b^{-qvm}s^{-um}(\alpha_1x+ \alpha_2x^q+\cdots+\alpha_nx^{q^{n-1}})^{um}\right).
	\end{multline*}
	This completes the proof. 
	
\end{proof}

\section*{Declarations}

\begin{itemize}
	\item
	The works is partially supported by the National Natural Science Foundation of China (Grant Nos. 12171163 and 11901083 ) and  the Guangdong Basic and Applied Basic Research Foundation (Grant Nos. 2020A1515111090 and 2024A1515010589).
\end{itemize}

\bibliography{sn-bibliography}

\end{document}